\documentclass[12pt]{article}

\usepackage{latexsym}
\usepackage{amssymb}
\usepackage{graphicx}
\usepackage{color}
\usepackage{amsmath}
\usepackage{mathrsfs} %fonts veramente calligrafici con \mathscr{B}
\usepackage{booktabs}

\newtheorem{theorem}{Theorem}

\newtheorem{corollary}[theorem]{Corollary}

\newtheorem{definition}[theorem]{Definition}

\newtheorem{lemma}[theorem]{Lemma}

\newenvironment{proof}[1][Proof]{\textbf{#1.} }{\ \rule{0.5em}{0.5em}}
\newcommand{\kom}[1]{}
\renewcommand{\kom}[1]{{\bf [#1]}}
\definecolor{blau}{rgb}{0.1,0.0,0.9}

\newcounter{komcounter}
\numberwithin{komcounter}{section}

\title{Sharp decay estimates for subsolutions of fully
nonlinear equations in thin domains}
\date{}
\author{JOHN FABRICIUS {\small{$^1$}} and NIKLAS L.P. LUNDSTR\"{O}M {\small{$^2$}}}

\date{\today}

\begin{document}

\maketitle

\begin{center}
{\small{$^1$}}
{\it \small Department of Engineering Sciences and Mathematics},\\ 
{\it \small Lule{\aa} University of Technology, SE-97187 Lule{\aa}, Sweden\/{\rm ;}}\\  
{\it \small john.fabricius@ltu.se}\\
{\small{$^2$}}
{\it \small Department of Mathematics and Mathematical Statistics},\\
{\it \small Ume{\aa} University, SE-90187 Ume{\aa}, Sweden\/{\rm ;}}\\
{\it \small niklas.lundstrom@umu.se}\\
\end{center}

\maketitle

%%%%%%%%%%%%%%%%%%%%%%%%%%%%%%%%%%%%%%%%%%%%
%%%%%%%%%%%%%%%%%%%%%%%%%%%%%%%%%%%%%%%%%%%%
%%%%%%%%%%%%%%%%%%%%%%%%%%%%%%%%%%%%%%%%%%%%
%%%%%%%%%%%%%%%%%%%%%%%%%%%%%%%%%%%%%%%%%%%%
%%%%%%%%%%%%%%%%%%%%%%%%%%%%%%%%%%%%%%%%%%%%
%%%%%%%%%%%%%%%%%%%%%%%%%%%%%%%%%%%%%%%%%%%%
%%%%%%%%%%%%%%%%%%%%%%%%%%%%%%%%%%%%%%%%%%%%
%%%%%%%%%%%%%%%%%%%%%%%%%%%%%%%%%%%%%%%%%%%%
%%%%%%%%%%%%%%%%%%%%%%%%%%%%%%%%%%%%%%%%%%%%
%%%%%%%%%%%%%%%%%%%%%%%%%%%%%%%%%%%%%%%%%%%%

\begin{abstract}
We reduce the problem of establishing pointwise upper decay estimates in thin domains for viscosity subsolutions of fully nonlinear partial differential equations to solving a one-dimensional differential inequality. The resulting estimates make explicit the interaction between the transverse geometry of the domain and the ellipticity of the equation, and are shown to be sharp. Our framework allows the ellipticity to degenerate at infinity and accommodates nonhomogeneous equations with lower-order terms. As an application, we obtain decay estimates for $p$-subharmonic functions when $1<p\leq 2$.\\

%
% 35B53: "Liouville theorems and Phragmén-Lindelöf theorems in context of PDEs" 
%
% 35J25: "Boundary value problems for second-order elliptic equations" 
%
% 35J60: "Nonlinear elliptic equations" 
%
% 35B50: "Maximum principles in context of PDEs"
%
% 35J70: "Degenerate elliptic equations" 
%
% 35J70: "Degenerate elliptic equations" (MSC2020)
%
% 42B37: "Harmonic analysis and PDEs" (MSC2020)
%

\end{abstract}

%%%%%%%%%%%%%%%%%%%%%%%%%%%%%%%%%%%%%%%%%%%%
%%%%%%%%%%%%%%%%%%%%%%%%%%%%%%%%%%%%%%%%%%%%
%%%%%%%%%%%%%%%%%%%%%%%%%%%%%%%%%%%%%%%%%%%%
%%%%%%%%%%%%%%%%%%%%%%%%%%%%%%%%%%%%%%%%%%%%
%%%%%%%%%%%%%%%%%%%%%%%%%%%%%%%%%%%%%%%%%%%%
%%%%%%%%%%%%%%%%%%%%%%%%%%%%%%%%%%%%%%%%%%%%
%%%%%%%%%%%%%%%%%%%%%%%%%%%%%%%%%%%%%%%%%%%%
%%%%%%%%%%%%%%%%%%%%%%%%%%%%%%%%%%%%%%%%%%%%
%%%%%%%%%%%%%%%%%%%%%%%%%%%%%%%%%%%%%%%%%%%%
%%%%%%%%%%%%%%%%%%%%%%%%%%%%%%%%%%%%%%%%%%%%

\section{Introduction}

\setcounter{theorem}{0}
\setcounter{equation}{0}

\noindent
The classical Phragm\'{e}n--Lindelöf principle (PLP) gives decay estimates for subharmonic functions defined in a positive half-strip
\[
\Omega=\{(x_1,z)\in \mathbf R^2: x_1>0,\:\lvert z\rvert <M\}.
\]
Suppose %$u\in C(\overline{\Omega})$ 
$u$ is bounded, subharmonic in $\Omega$ and vanishes on $z=\pm M.$
%\begin{enumerate}
%\item[(i)] harmonic in $S$ and continuous on $\overline{S}$,
%\item[(ii)] bounded in the strip,
%\item[(iii)] vanishes on the lateral boundary, i.e. $u(x_1,\pm M)=0.$
%\end{enumerate}
Then $u$ decays exponentially as $x_1\to \infty.$
We prove this using an elementary barrier argument that already contains the main idea of the paper. 
For a suitable constant $A$, the function
$$
\hat u_L(x_1,z) = A \left(e^{- \sqrt\kappa x_1} + \frac{x_1}{L}\right) \cos\left( \sqrt\kappa z\right) 
\quad \text{with}\quad \kappa = \frac{\pi^2}{4 M^2}
$$
is strictly superharmonic in $\Omega_L = \Omega \cap\{1 < x_1 < L\}$ for each $L>1$. 
Thus, if it dominates $u$ on $\partial \Omega_L$, 
then by a well-known comparison principle it barriers $u$ from above in $\Omega_L$, and
taking the infimum over $L$ then implies the decay, since $\Omega_L \to \Omega\cap\{1 < x_1\}$ as $L \to \infty$ and
\begin{align}\label{eq:inf-intro}
u (x_1, z) \leq \inf_{L > x_1} \hat u_L(x_1,z) \leq A e^{- \sqrt\kappa x_1} \quad \text{whenever} \quad x_1 > 1, |z| < M.
%\text{in} \quad \Omega\cap\{1 < x_1\}.
\end{align}
Indeed, away from the inlet at $x_1 = 0$, it is well known that any subharmonic function with $u \leq 0$ on $|z| = M$ must vanish no slower than the distance to this part of the boundary (given by $|z| = M$), and thus $\hat u_L$ dominates $u$ on $\partial \Omega_L$ in the sense that for $A$ large enough, $\hat u_L \geq u$ on $\partial \Omega_L$.
Hence $u \leq  \hat u_L$ in $\Omega_L$, for any arbitrary $L > 1$. 
Therefore, we conclude validity of \eqref{eq:inf-intro} implying the exponential decay.
The rate is optimal, since the principal mode
$$
u(x_1,z) = e^{- \sqrt\kappa x_1} \cos\left( \sqrt\kappa z\right) 
$$
is a positive harmonic function in $\Omega$ that vanishes on its lateral boundary.
%As the harmonic function
%
%$$
%v(x_1,z) = e^{- \sqrt\kappa x_1} \cos\left( \sqrt\kappa z\right) 
%$$
%
%satisfies 0 on $\partial S \cap \{|z| = M\} \cap \{x_1 > 0\}$, is positive in $S$ and dying exponentially with exponent 5$\sqrt\kappa$, this exponent is sharp.

In this paper, we extend the above argument to viscosity subsolutions of fully nonlinear, possibly nonhomogeneous and degenerate elliptic equations in a broad class of thin domains which we call $(m,M)$-thin domains (see Definition \ref{def:narrowdomain}).  
%In this paper, we expand on the above argument to prove pointwise decay estimates for viscosity subsolutions of fully nonlinear nonhomogeneous degenerate elliptic PDEs in a class of thin domains which we call $(m,M)$-thin domains (see Definition \ref{def:narrowdomain}).  
In particular, we consider equations in nondivergence form,
\begin{align}\label{eq:main}
F(x,u,Du,D^2u) = 0,\tag{$\star$}
\end{align}
in $(m,M)$-thin domains $\Omega \subset \mathbb{R}^n$ for $n \geq 2$.
Here, $Du$ is the gradient and $D^2u$ the Hessian.
$F:\mathbb{R}^n \times \mathbb{R} \times \mathbb{R}^n  \times \mathbb{S}^n \rightarrow \mathbb{R}$ in which $\mathbb{S}^n$ is the set of symmetric $n \times n$ matrices equipped with the positive semi-definite ordering;
for $X, Y \in \mathbb{S}^n$, we write $X \leq Y$ if $\langle (X - Y) \xi, \xi \rangle \leq 0$ for all $\xi \in \mathbb{R}^n$.
We assume
\begin{itemize}
\item[] Monotonicity; $F(x,u,p,X) \geq  F(x,v,p,Y)$
whenever $u \geq v$, $X \leq Y$, and the growth condition
\begin{align}\label{eq:ass_drift_super}
F(x,r,p,X) \geq  \mathcal P_{\lambda,\Lambda}(p,X) 
%- \Lambda \text{Tr}(X^+) + \lambda \text{Tr}(X^-) 
- \Phi(x,r,p)
\tag{$\star\star$}
\end{align}
where $\mathcal P_{\lambda,\Lambda}$ is a Bellman-type operator (see Section \ref{sec:prel}) 
%Write $x=(x_1,y,z)$, where $y\in\mathbb R^{n-m-1}$ and $z\in\mathbb R^m$. 
%For $X\in\mathbb S^n$, let $X_y$ and $X_z$ denote the blocks corresponding to the variables $y$ and $z$, then
%
%\begin{align*}%\label{eq:ass_drift_super}
%\mathcal P^-_{m}(X)
%= - \Lambda\left( X_{11}^+ + \text{Tr} X_{y}^+  + \left(\text{Tr} X_{z}\right)^+ \right) 
%+ \lambda\left( X_{11}^- + \text{Tr} X_{y}^- +\left(\text{Tr} X_{z}\right)^- \right) 
%\end{align*}
%
%in which $X_y = X_y^+ - X_y^-$ with $X_y^+ \geq 0$, $X_y^- \geq 0$, and $X_y^+  X_y^- = 0$, 
in which $0 < \lambda \leq \Lambda < \infty$  may depend on $x,r,p$ so that $ \lambda, \Lambda, \Phi : \mathbb{R}^n \times \mathbb{R} \times \mathbb{R}^n \to \mathbb{R}$.
\end{itemize}

Our main results establish general decay estimates for viscosity subsolutions of \eqref{eq:main}, assuming \eqref{eq:ass_drift_super}, in $(m,M)$-thin domains in $\mathbb{R}^n$ (Theorem \ref{th:exp-decay-pipe}). 
These results apply to equations with vanishing ellipticity ($\frac{\lambda}{\Lambda} \rightarrow 0$) and with lower-order terms (nonhomogeneous case, $\Phi \neq 0$). 
If $\Phi \equiv 0$ and if subsolutions satisfy a boundary growth estimate, 
then we obtain a slightly sharper result (Corollary \ref{cor:exp-decay-pipe}). 

More precisely, our main result reduces the decay problem to finding a positive function $\widehat\Phi$ of the longitudinal coordinate $x_1$ that satisfies a differential inequality. 
Indeed, if $\widehat\Phi$ satisfies 
\begin{align}\label{eq:introineq}
    \frac{\widehat\Phi''(x_1)}
         {\widehat\Phi(x_1)}
    \leq \kappa\,
    \frac{\lambda(x_1)}
         {\Lambda(x_1)}
     \quad \text{then} \quad u(x) \leq A \widehat\Phi(x_1),
\end{align}
whenever $x = (x_1,x_2,\dots,x_n)$ is in an $(m,M)$-thin domain, and where $\kappa = \kappa(m,M)$ 
%
%\begin{align*}
%    \kappa
%    =
%    \left(
%        \frac{j_{\frac{m-2}{2},1}}{M}
%    \right)^2
%\end{align*}
%
is the first Dirichlet eigenvalue of $-\Delta$ in the $m$-dimensional ball of radius $M$. 
See \eqref{eq:kappa} and Section \ref{sec:prel} for an explicit expression of $\kappa$ and how it relates to the geometry of the domain.  
%Here, $m$ and $M$ characterize the $(m,M)$-thin domain $\Omega$ and 
%$j_{\nu,1}$ is the first positive zero of the Bessel function of the first kind of order $\nu$.
For $m=1$ we obtain
$
\kappa = \frac{\pi^2}{4M^2} 
$
which agrees with the introductory example above. %, see also Table \ref{tab:table1}.

The estimate makes explicit the interaction between the transverse
geometry of the domain via $\kappa$ 
and the possible degeneration of the ellipticity, $\lambda/\Lambda \to 0$, 
in the longitudinal direction $x_1$. 
For uniformly elliptic equations, this immediately gives exponential
decay; any subsolution must decay at least as fast as
\begin{align*}
    \exp\left(-\sqrt{\kappa\,\frac{\lambda}{\Lambda}}\,x_1\right).
\end{align*}
On the other hand, if the ellipticity ratio
$\lambda/\Lambda$ tends to zero, the one-dimensional inequality produces
slower decay rates and may eventually cease to admit a positive function
tending to zero. 
Moreover, whenever \eqref{eq:introineq} admits a positive solution with equality, 
tending to zero as $x_1 \to \infty$, our estimate is attained by an explicit
solution of an equation in the class of PDEs considered.
Thus, in such a case, both the dependence on the ellipticity ratio $\lambda/\Lambda$
and the transverse constant $\kappa$ are optimal; see Section~\ref{sec:sharp}.

%%%%%%%%%%%%%%%%%%%%%%%%%%%%%%%%%%%%%%%%%%%%%%%%%%%%%%%%%%%%%%%%%
%%%%%%%%%%%%%%%%%%%%%%%%%%%%%%%%%%%%%%%%%%%%%%%%%%%%%%%%%%%%%%%%%
%%%%%%%%%%%%%%%%%%%%%%%%%%%%%%%%%%%%%%%%%%%%%%%%%%%%%%%%%%%%%%%%%
%%%%%%%%%%%%%%%%%%%%%%%%%%%%%%%%%%%%%%%%%%%%%%%%%%%%%%%%%%%%%%%%%
%%%%%%%%%%%%%%%%%%%%%%%%%%%%%%%%%%%%%%%%%%%%%%%%%%%%%%%%%%%%%%%%%

\subsection*{Literature review}

The half-strip example above also admits a useful interpretation of PLP in terms of the longitudinal propagation of boundary disturbances. We remark that the decay rate of the solution is sensitive to the boundary condition on $z=\pm M$. Under homogeneous Dirichlet conditions, every transverse mode has positive eigenvalue and therefore decays exponentially. Under homogeneous Neumann conditions, a constant zero mode remains, while all nonconstant modes still decay.

In other words, changing the data at the inlet $x_1=0$ has an exponentially small effect far down the strip, provided growing modes are excluded by an appropriate boundedness condition. This principle has been generalized in various directions. An abstract generalization is due to Lax \cite{L57} that extends PLP from harmonic functions to higher-order elliptic equations in half-cylinders. While the present work establishes pointwise decay estimates, other formulations measure decay through a Dirichlet energy integral over tails of the domain \cite{Ole80}. In the context of elliptic systems arising in continuum mechanics, estimates of this kind are commonly known as Saint-Venant-type estimates, see e.g. \cite[Ch.~XIII]{G11} and the references therein. Their connection to boundary layers is discussed below.

Results closely related to the present work are the Phragm\'en--Lindel\"of estimates in \cite{L22} and the maximum and Phragm\'en--Lindel\"of principles in \cite{CDV19}. In \cite{L22}, sharp lower growth estimates for subsolutions of fully nonlinear equations in half-spaces are characterized through one-dimensional differential inequalities. Here we consider the complementary problem of obtaining upper decay estimates in thin domains. In \cite{CDV19}, the authors study fully nonlinear degenerate elliptic equations in cylindrical domains under structural assumptions that relate strict ellipticity to the bounded directions while permitting degeneracy in the unbounded ones. Their results identify growth conditions under which boundary sign information propagates throughout the domain. By contrast, we assume boundedness and lateral boundary control and derive explicit pointwise decay estimates along a distinguished longitudinal direction. These estimates depend sharply on the transverse Dirichlet eigenvalue and remain applicable when the ellipticity ratio vanishes at infinity.

Related asymptotic and decay problems in cylindrical domains have been studied in several other settings. Linear and nonlinear elliptic problems in domains becoming unbounded in one or more directions were considered in \cite{CR02}, while \cite{C07} established exponential decay for nonlinear elliptic systems in infinite cylinders, with applications to boundary layer problems. For Laplacian eigenfunctions in domains with branches of variable cross-section, exponential decay is governed by spectral properties of the transverse geometry \cite{DNG12}. Phragm\'en--Lindel\"of-type estimates for certain nonlinear elliptic equations in circular cylinders were obtained in \cite{SR86}. The present framework allows fully nonlinear equations in nondivergence form with degenerating ellipticity and lower-order terms, and makes the interaction among ellipticity, transverse geometry, and longitudinal decay explicit through a one-dimensional differential inequality.

\subsubsection*{Boundary layers}

Boundary layers are a central topic in singular perturbation theory \cite{Hol13}. In a second-order Dirichlet problem the highest derivative is multiplied by a small parameter $\epsilon.$ Setting $\epsilon$ to zero yields the so-called outer solution of the problem, which is unable to satisfy both boundary values. Rescaling the coordinates near the boundary gives a complementary inner solution which represents a rapid transition of the true solution in a small region, called a boundary layer, where the outer solution fails. A boundary layer corrector matches the inner and outer solutions near the boundary into a uniformly valid approximation. To remain localized near the boundary, the corrector must decay along the corresponding stretched semi-infinite domain, typically at an exponential rate. Establishing this decay is therefore a Phragm\'en--Lindel\"of-type problem.

Boundary layers are particularly important in the analysis of fluid flow in thin domains \cite{Pan98,Pan24}. A representative example comes from pressure-driven Stokes flow in a thin channel or tube \cite{BFMW25}. The leading-order Poiseuille profile satisfies the no-slip condition on the lateral walls but is generally incompatible with realistic inlet and outlet conditions. Boundary layer correctors posed on semi-infinite tubes compensate for this mismatch. Their exponential decay ensures that the correction remains localized near the ends and that the Poiseuille profile is accurate in the interior. Although the Stokes system lies outside the scalar framework of the present paper, it illustrates the same decay mechanism. 

In linear elasticity it is important to understand how disturbances in the traction prescribed at the end of a long or thin body propagate into its interior. Saint-Venant's principle states that two statically equivalent loads, i.e. tractions which produce the same resultant force and moment, generate displacement and stress fields that become increasingly similar with distance from the loaded region \cite{Toup65,Know66}. Equivalently, the perturbation of a self-equilibrated load decays rapidly away from its support, so the region near the loaded end is effectively a boundary layer. Similar localization phenomena arise in homogenization problems \cite{All99}.

The remainder of the paper is organized as follows.
Section~\ref{sec:prel} introduces the necessary definitions and preliminary results.
Section~\ref{sec:main} presents the main theorems and their proofs.
Finally, Section~\ref{sec:sharp} discusses sharpness and provides examples and applications of the estimates.

%%%%%%%%%%%%%%%%%%%%%%%%%%%%%%%%%%%%%%%%%%%%%%%%%%%%%%%%%%%%%%%%%%%%%%%%%%%%
%%%%%%%%%%%%%%%%%%%%%%%%%%%%%%%%%%%%%%%%%%%%%%%%%%%%%%%%%%%%%%%%%%%%%%%%%%%%
%%%%%%%%%%%%%%%%%%%%%%%%%%%%%%%%%%%%%%%%%%%%%%%%%%%%%%%%%%%%%%%%%%%%%%%%%%%%
%%%%%%%%%%%%%%%%%%%%%%%%%%%%%%%%%%%%%%%%%%%%%%%%%%%%%%%%%%%%%%%%%%%%%%%%%%%%
%%%%%%%%%%%%%%%%%%%%%%%%%%%%%%%%%%%%%%%%%%%%%%%%%%%%%%%%%%%%%%%%%%%%%%%%%%%%

\section{Preliminaries}
\label{sec:prel}
\setcounter{theorem}{0}
\setcounter{equation}{0}

In the following, 
a connected open set will be referred to as a domain, 
and we will denote points in Euclidean $n$-space $\mathbb{R}^n$  by $x=(x_1,y,z) = (x_1,y_1,\dots,y_\ell,z_1,\dots,z_m)$, where $y\in\mathbb{R}^{\ell}$, $z\in\mathbb{R}^{m}$ and $n = 1 + \ell + m$.  
We assume $\ell\geq 0$, $m\geq 1$ and define:

\begin{definition}\label{def:narrowdomain}
We say that a domain $\Omega \subset \mathbb{R}^n$ is $(m,M)$-thin if
\begin{align*}%\label{eq:domain_bound}
\Omega \subset \{x = (x_1,y_1,\dots,y_\ell,z_1,\dots,z_m) \in \mathbb{R}^n : x_1 > 0, |z| < M\}.
\end{align*}
Moreover, for an $(m,M)$-thin domain $\Omega$ we use the notation $\Sigma_L = \overline \Omega \cap \{x : x_1 = L\}$ for $L \geq 0$, and $\Gamma^+ = \partial \Omega \cap \{x_1 > 0\}$ so that $\partial \Omega = \Sigma_0 \cup \Gamma^+$. Finally, we call the domain 
$
\{x=(x_1,y,z)\in\mathbb R^n:x_1>0,\ |z|<M\}
$
the maximal $(m,M)$-thin domain.
\end{definition} 
We remark that a $(2,M)$-thin domain in $\mathbb{R}^3$ may model a pipe contained in a cylinder with radius $M$ having `inlet' $\Sigma_0$ at $x_1 = 0$ and lateral boundary $\Gamma^+$, 
and a $(1,M)$-thin domain in $\mathbb{R}^3$ is trapped between two parallel planes at distance $2M$, having `inlet' $\Sigma_0$ at $x_1 = 0$ and, if it is maximal, lateral boundary $\Gamma^+$ consisting of the planes $z = \pm M$. 

To define the operator $\mathcal P_{\lambda,\Lambda}$, 
let $\lambda, \Lambda : \mathbb{R}^n \times \mathbb{R} \times \mathbb{R}^n \to \mathbb{R}$, $0<\lambda\leq\Lambda$. 
For $x \in \mathbb R^n, r \in \mathbb{R}, q\in\mathbb R^n$, define
\begin{align*}
\mathcal A(x,r,q) = &
\left\{
\begin{pmatrix}
a&0&0\\
0&B&0\\
0&0&cI_m
\end{pmatrix}
:
\lambda\leq a,c\leq\Lambda,\quad
\lambda I_\ell\leq B\leq\Lambda I_\ell
\right\}
\\
&\cup
\left\{
\Lambda I_n-(\Lambda-\lambda)
\frac{q\otimes q}{|q|^2 + \eta^2} : \eta > 0
\right\},    
\end{align*}
where $\otimes$ denotes the outer product. 
The associated Bellman-type operator considered in this paper is
\[
\mathcal P_{\lambda,\Lambda}(q,X)
=
\inf_{A\in\mathcal A(x,r,q)}
\bigl\{-\operatorname{Tr}(AX)\bigr\},
\quad \text{whenever} \quad X\in\mathbb S^n.
\]
Let \(X_y\) and \(X_z\) denote the diagonal blocks of \(X\)
corresponding to the \(y\)- and \(z\)-variables, respectively.
Write \(X_y=X_y^+-X_y^-\) for its spectral decomposition into
positive and negative parts, and set \(t^\pm=\max\{\pm t,0\}\)
for \(t\in\mathbb R\). Then

\begin{align}\label{eq:operatorexpression}
\mathcal P_{\lambda,\Lambda}(q,X)=
\begin{cases}
\displaystyle
\min\left\{
B_{\lambda,\Lambda}(X),\,
-\Lambda\operatorname{Tr}X
+(\Lambda-\lambda)\frac{\langle Xq,q\rangle}{|q|^2}
\right\},&q\neq0,\\[2mm]
B_{\lambda,\Lambda}(X),&q=0.
\end{cases}
\end{align}
where
\begin{align*}
B_{\lambda,\Lambda}(X)=
-\Lambda\left(X_{11}^+
+\operatorname{Tr}(X_y^+)+(\operatorname{Tr}X_z)^+\right)
+\lambda\left(X_{11}^-
+\operatorname{Tr}(X_y^-)+(\operatorname{Tr}X_z)^-\right).
\end{align*}

We let $\overline{E}$ be the closure and $\partial E$ the boundary of the set $E \subset \mathbb{R}^n$,  
and $d(y, E)$ equals the distance from $y \in \mathbb{R}^n$ to $E$. 
The space of continuous functions on $E$ is denoted by $C(E)$; 
%$C^{\alpha}(E)$, for $0<\alpha\leq 1$, consists of functions in $C(E)$ that are Hölder continuous of order $\alpha$,
$C^k(E)$ is the space of functions whose $k$:th partial derivatives exist and are continuous on $E$. 
%, and $C^{1,\alpha}$ consists of functions in $C^1(E)$  whose first-order derivatives are also Hölder continuous of order $\alpha$. 
$LSC(E)$ denotes lower, while $USC(E)$ denotes upper semicontinuous functions on $E$, 
$\langle .,.\rangle$ denotes the standard Euclidean inner product on $\mathbb{R}^n$ and $|x|=\sqrt{\langle x, x \rangle}$ is the Euclidean norm of $x$. 
The open ball centered at $x$ with radius $r>0$ is denoted by $B(x,r)=\{y\in\mathbb{R}^n:|x-y|<r\}$, 
$dx$ is the Lebesgue $n$-measure on $\mathbb{R}^n$, 
and, %As a convention, in this paper, 
unless otherwise
stated, $c$ will denote a constant $\geq 1$, not necessarily the same at each occurrence.
%, and we write $A \lesssim B$ ($A \gtrsim B$) if there exists $c$ such that $A \leq c B$ ($c A \geq B$), 
%and $A \approx B$ means there exists $c$ so that $c^{-1} A \leq B \leq c A$. 
%
Let $J_\nu$ denote the Bessel function of the first kind of order $\nu$, 
$j_{\nu,1}$ the first positive zero of $J_\nu$, and let
\begin{align}\label{eq:kappa}
    \kappa  =   \left(\frac{j_{\frac{m-2}{2},1}}{M}\right)^2
\end{align}
denote the first Dirichlet eigenvalue of $-\Delta$ in the $m$-dimensional ball $B(0,M)$. 
Table~\ref{tab:table1} gives $j_{\frac{m-2}{2},1}$ and $\kappa$ for some values of $m$. 
For some standard estimates of $\kappa$ and $j_{\nu,1}$ we refer the reader to e.g. \cite{bok_spec_func}. 

\begin{table}[ht]
    \centering
    \renewcommand{\arraystretch}{2.5}
    \setlength{\tabcolsep}{16pt}
    \begin{tabular}{c c c c}
        \hline
        \medskip
        $m$ & $\nu$ & $j_{\nu,1}$ & $\displaystyle \kappa
        = \left(\frac{j_{\nu,1}}{M}\right)^2$ \\
        \hline
        1 & $\displaystyle -\frac{1}{2}$ & $\displaystyle \frac{\pi}{2}$
          & $\displaystyle \frac{\pi^2}{4M^2}$ \\

        2 & $0$ & $2.40482556$
          & $\displaystyle \frac{5.7832}{M^2}$ \\

        3 & $\displaystyle \frac{1}{2}$ & $\pi$
          & $\displaystyle \frac{\pi^2}{M^2}$ \\

        4 & $1$ & $3.83170597$
          & $\displaystyle \frac{14.682}{M^2}$ \\

        5 & $\displaystyle \frac{3}{2}$ & $4.49340946$
          & $\displaystyle \frac{20.191}{M^2}$ \\

        6 & $2$ & $5.13562230$
          & $\displaystyle \frac{26.374}{M^2}$ \medskip\\
        \hline
    \end{tabular}
    \caption{The first positive zero $j_{\nu,1}$ and the corresponding
    first Dirichlet eigenvalue $\kappa$ for selected dimensions $m$.}
    \label{tab:table1}
\end{table}

A function $u : \Omega \to \mathbb{R}$ is a classical subsolution (supersolution) to \eqref{eq:main} in $\Omega$ if it is twice differentiable in $\Omega$ and satisfies $F(x,u,Du,D^2u) \leq 0$ ($F(x,u,Du,D^2u) \geq 0$).
If the inequality is strict then $u$ is a strict classical subsolution (supersolution),
and if equality holds then it is a classical solution.
We choose to present our main results for viscosity subsolutions, defined in a suitable way depending on $F$, see e.g. \cite{CIL92} for continuous $F$ and \cite{JJ12} for definitions of viscosity solutions for $p$-Laplace-type equations.
The use of viscosity solutions is not essential to our results. Other notions of weak solution may be adopted when they are better suited to the equation under consideration. The key ingredient in our argument is the weak comparison principle stated in Lemma~\ref{le:comp-weak-some-sence}. Thus, a subsolution may be understood in any sense for which this principle holds against the strict classical supersolutions constructed in the next section.
\begin{lemma}\label{le:comp-weak-some-sence}
Let $\Omega$ be a bounded domain and assume the monotonicity in \eqref{eq:ass_drift_super}. Let $u \in USC(\overline{\Omega})$ be a subsolution (in some weak sense)
and $v$ a classical strict supersolution of \eqref{eq:main} in $\Omega$, continuous on $\overline{\Omega}$.
If $u \leq v$ on $\partial \Omega$ then $u < v$ in $\Omega$.
\end{lemma}

\noindent
\begin{proof}
Assume that $v$ is a classical strict supersolution and also, by way of contradiction, that $u \geq v$ somewhere in $\Omega$. 
By $USC$ the function $u-v$ then attains a maximum $\geq 0$ at some point $x_0 \in \Omega$.
Since $v \in C^2(\Omega)$, $u - v$ has a maximum at $x_0$, and $u$ is a viscosity subsolution, it follows by definition that
\begin{align}\label{eq:johej-lemma1}
F(x_0, u(x_0), Dv(x_0), D^2 v(x_0)) \leq 0.
\end{align}
But since $v$ is a classical strict supersolution, 
we have 
$
F(x,v(x), Dv(x), D^2v(x)) > 0 %\quad \text{whenever} \quad x \in \Omega,
$
whenever $x \in \Omega$, 
and as $u(x_0) \geq v(x_0)$ it follows from the monotonicity in \eqref{eq:ass_drift_super} that
$$
F(x_0, u(x_0), Dv(x_0), D^2 v(x_0)) \geq F(x_0, v(x_0), Dv(x_0), D^2 v(x_0)) > 0.
$$
This contradicts \eqref{eq:johej-lemma1} and hence we have proved the lemma. 
\end{proof}

%%%%%%%%%%%%%%%%%%%%%%%%%%%%%%%%%%%%%%%%%%%%%%%%%%%%%%%%%%%%%%%%%%%%%%%%%%%%
%%%%%%%%%%%%%%%%%%%%%%%%%%%%%%%%%%%%%%%%%%%%%%%%%%%%%%%%%%%%%%%%%%%%%%%%%%%%
%%%%%%%%%%%%%%%%%%%%%%%%%%%%%%%%%%%%%%%%%%%%%%%%%%%%%%%%%%%%%%%%%%%%%%%%%%%%
%%%%%%%%%%%%%%%%%%%%%%%%%%%%%%%%%%%%%%%%%%%%%%%%%%%%%%%%%%%%%%%%%%%%%%%%%%%%
%%%%%%%%%%%%%%%%%%%%%%%%%%%%%%%%%%%%%%%%%%%%%%%%%%%%%%%%%%%%%%%%%%%%%%%%%%%%
%%%%%%%%%%%%%%%%%%%%%%%%%%%%%%%%%%%%%%%%%%%%%%%%%%%%%%%%%%%%%%%%%%%%%%%%%%%%
%%%%%%%%%%%%%%%%%%%%%%%%%%%%%%%%%%%%%%%%%%%%%%%%%%%%%%%%%%%%%%%%%%%%%%%%%%%%
%%%%%%%%%%%%%%%%%%%%%%%%%%%%%%%%%%%%%%%%%%%%%%%%%%%%%%%%%%%%%%%%%%%%%%%%%%%%
%%%%%%%%%%%%%%%%%%%%%%%%%%%%%%%%%%%%%%%%%%%%%%%%%%%%%%%%%%%%%%%%%%%%%%%%%%%%
%%%%%%%%%%%%%%%%%%%%%%%%%%%%%%%%%%%%%%%%%%%%%%%%%%%%%%%%%%%%%%%%%%%%%%%%%%%%
%%%%%%%%%%%%%%%%%%%%%%%%%%%%%%%%%%%%%%%%%%%%%%%%%%%%%%%%%%%%%%%%%%%%%%%%%%%%

\section{Statement and proof of the main theorems}
\label{sec:main}
\setcounter{theorem}{0}
\setcounter{equation}{0}

We are now ready to state our main results, beginning with the following theorem.

\begin{theorem} \label{th:exp-decay-pipe}
Let $\Omega$ be an $(m,M)$-thin domain,  
assume that \eqref{eq:ass_drift_super} holds with 
$\lambda=\lambda(x_1), \Lambda = \Lambda(x_1)$, %and 
%$\lambda/\Lambda$ nonincreasing, 
and
$$
\Phi(x,r,p) \leq C \widehat \Phi(x_1) \lambda(x_1) \quad \text{whenever} \quad (x,r,p) \in \Omega\times \mathbb{R} \times \mathbb{R}^n
$$
for some positive, nonincreasing $\widehat\Phi \in C^2([0,\infty))$
satisfying $\widehat \Phi'' > 0$ on $[0,\infty)$ and $\widehat \Phi(0) = 1$.
If $u$ is a viscosity subsolution of \eqref{eq:main} in $\Omega$, 
satisfying $u \leq C_1$ on $\overline \Omega$ and $u \leq 0$ on $\Gamma^+$,  
then for every $\bar \epsilon > 0$ there exists $A = A(\bar \epsilon, C, C_1, m, M) > 0$ such that
\begin{align*}
\frac{\widehat\Phi''(x_1)}{\widehat\Phi (x_1)} \leq 
\left(\kappa - \bar \epsilon \right) \frac{\lambda(x_1)}{\Lambda(x_1)}  \quad \text{for all}\quad x_1 > 0
\quad \implies 
\quad  u(x) \leq A\, \widehat \Phi(x_1) %\quad \text{in}\quad \Omega,
\end{align*}
whenever $x \in  \Omega$ and 
where $\kappa$ is as in \eqref{eq:kappa}.% = \left(\frac{j_{\frac{(m-2)}{2}, 1} }{M}\right)^2$.
\end{theorem}

If, in addition to the assumptions of Theorem \ref{th:exp-decay-pipe}, 
the subsolutions satisfy a boundary growth estimate, 
then, in cases of non-positive lower-order terms  ($\Phi \leq 0$), 
we can take $\bar \epsilon = 0$ in the conclusion 
and prove Corollary \ref{cor:exp-decay-pipe} below. 
A version of such boundary estimate, tailored for our needs, 
can be stated as follows:
Let $\Omega \subset \mathbb{R}^n$ be a maximal $(m,M)$-thin domain, 
$w \in \partial\Omega$, $0 < r < r_0$, 
and $u$ be a subsolution of \eqref{eq:main} in $B(w,r)\cap \Omega$ 
satisfying $u = 0$ on $\partial \Omega \cap B(w,r)$.
Then there exists $\bar c > 0$, independent of $w$ and $r$, so that
\begin{align}\label{eq:for_corr} 
 u(x) \leq \bar c\, d(x, \partial \Omega) \quad \text{whenever} \quad x \in \Omega \cap B(w,r/\bar c).
\end{align}
We remark that maximal $(m,M)$-thin domains are convex and thus satisfy an exterior sphere condition. 
Therefore, the necessary estimate can be found for a wide class of PDEs in \cite{LOT24,LOS25}.

\begin{corollary}\label{cor:exp-decay-pipe}
Let $\Omega$, $\lambda/\Lambda$, $\widehat\Phi$, $u$ and $\kappa$ be as in 
Theorem \ref{th:exp-decay-pipe} and assume \eqref{eq:ass_drift_super} with $\Phi \leq 0$. 
Suppose also that the boundary growth estimate in \eqref{eq:for_corr} holds for $u$. 
Then there exists $A = A(\bar c, C_1, M, m) > 0$ such that
\begin{align*}
%1 \lesssim 
\frac{\widehat\Phi''(x_1)}{\widehat\Phi (x_1)} \leq  \kappa\, \frac{\lambda(x_1)}{\Lambda(x_1) } \quad \text{for all}\quad x_1 > 0
\quad \implies %\text{then}
\quad  u(x) \leq A\, \widehat \Phi(x_1)
\end{align*}
whenever $x \in  \Omega$. 
%and where $\kappa$ is as in \eqref{eq:kappa}.
\end{corollary}

Before we prove our main theorems, we establish the following auxiliary lemma. 

\begin{lemma}
\label{le:supersol}
Suppose that \eqref{eq:ass_drift_super} holds and that there are positive $C^2(E)$-functions $f = f(x_1)$, $g = g(y)$ and $h = h(z) = \bar h(|z|)$ in $E$ such that
$f'' \geq 0$, $D^2 g  \geq 0$, $\Delta h \leq 0$, $\bar h$ is non-increasing. If
\begin{align*}   
   - \frac{\Delta_z h}{h} > \frac{\Phi(x,\hat u, D\hat u)}{ \lambda\, \hat u} +\frac{\Lambda}{\lambda}  \left( \frac{f''}{f}  +  \frac{\Delta_y g}{g}  \right)\quad \text{in} \quad E,
\end{align*}
where $\hat u(x) = f(x_1) g(y) h(z)$, 
then $\hat u$ is a strict classical supersolution of \eqref{eq:main} in $E$. 
\end{lemma}

\noindent
\begin{proof}
We will prove that the two inequalities
\begin{align}\label{eq:helplemmaeq1}
   -& \Lambda\left( \left(D^2\hat u\right)_{11}^+ + \text{Tr}\left(D^2 \hat u\right)^+_y  +\left(\text{Tr}\left(D^2 \hat u\right)_z\right)^+ \right) \notag\\
&+ \lambda\left( \left(D^2\hat u\right)_{11}^- + \text{Tr}\left(D^2 \hat u\right)^-_y +\left(\text{Tr}\left(D^2 \hat u\right)_z\right)^- \right)  \notag\\
&\geq  -\Lambda   \left(f'' g h + f \Delta_y g h \right) -  \lambda f g \Delta_z h, 
\end{align}
and, whenever $|D\hat u| \neq 0$,  
\begin{align}\label{eq:helplemmaeq2}
- \Lambda \operatorname{Tr} D^2\hat u + (\Lambda-\lambda)
 \frac{\langle D^2\hat u D\hat u, D\hat u\rangle}{|D\hat u|^2}
 \geq  -\Lambda   \left(f'' g h + f \Delta_y g h \right) -  \lambda f g \Delta_z h
\end{align}
hold in $E$. 
The result then follows by \eqref{eq:ass_drift_super}, \eqref{eq:operatorexpression} and the inequality in the lemma because 
\begin{align*}
F(x,\hat u,D\hat u,D^2\hat u) &\geq \mathcal P_{\lambda,\Lambda}(D\hat u,D^2\hat u) - \Phi(x,\hat u, D\hat u)  \\
&\geq  -\Lambda   \left(f'' g h + f \Delta_y g h \right) -  \lambda f g \Delta_z h  - \Phi(x,\hat u, D\hat u) > 0,
\end{align*}
and thus $\hat u$ is a strict classical supersolution of \eqref{eq:main} in $E$.
By sign assumptions, we have
$$
\left(D^2 \hat u\right)^+_{11} = f'' g h \quad \text{and} \quad 
\left(D^2 \hat u\right)^-_{11} = 0,
$$
and since $D^2 g  \geq 0$ it also follows that
$$
\text{Tr}\left(D^2 \hat u\right)^+_y =  f \Delta g_y h \quad \text{and} \quad \text{Tr}\left(D^2 \hat u\right)^-_y = 0.
$$
Moreover, 
$$
\left(\text{Tr}\left(D^2 \hat u\right)_z\right)^+ = f g \left(\Delta_z h\right)^+ = 0 \quad \text{and} \quad \left(\text{Tr}\left(D^2 \hat u\right)_z\right)^- = -f g \Delta_z h.
$$
Insertion of these equalities in \eqref{eq:helplemmaeq1} establishes its inequality.
To prove \eqref{eq:helplemmaeq2} we first claim that 
\begin{align}\label{eq:claimforlemma}
\frac{\Delta_\infty\hat u}{|D\hat u|^2}
\geq %\Delta_z\hat u=
fg\,\Delta_z h \quad\text{in}\quad E,
\end{align}
in which $\Delta_\infty\hat u=\langle D^2\hat uD\hat u,D\hat u\rangle$, giving
\begin{align*}
-& \Lambda \operatorname{Tr} D^2\hat u + (\Lambda-\lambda)
 \frac{\langle D^2\hat u D\hat u, D\hat u\rangle}{|D\hat u|^2}
=
-\Lambda\Delta\hat u + (\Lambda-\lambda) \frac{\Delta_\infty\hat u}{|D\hat u|^2}\notag\\
&\geq 
-\Lambda\bigl(f''gh + fh\,\Delta_y g + fg\,\Delta_z h\bigr) + (\Lambda-\lambda)fg\,\Delta_z h\\
& =
 -\Lambda   \left(f'' g h + f \Delta_y g h \right) -  \lambda f g \Delta_z h,
\end{align*}
which proves \eqref{eq:helplemmaeq2}. 
To establish claim \eqref{eq:claimforlemma}, we write
\begin{align*}
\Delta_\infty \hat u
={}&\hat u_{x_1}^2\hat u_{x_1x_1}
+\sum_{i,j=1}^{\ell}\hat u_{y_i} \hat u_{y_j}\hat u_{y_i y_j}
+\sum_{i,j=1}^{m}\hat u_{z_i}\hat u_{z_j}\hat u_{z_i z_j}\\
&+2\sum_{i=1}^{\ell}\hat u_{x_1}\hat u_{y_i}\hat u_{x_1y_i}
+2\sum_{j=1}^{m}\hat u_{x_1}\hat u_{z_j}\hat u_{x_1z_j} 
+2\sum_{i=1}^{\ell}\sum_{j=1}^{m}
  \hat u_{y_i}\hat u_{z_j}\hat u_{y_i z_j}\\
={}&f''(f')^2g^3h^3
+f^3h^3\langle D_y^2g\,D_yg,D_yg\rangle
+f^3g^3\langle D_z^2h\,D_zh,D_zh\rangle\\
&+2f(f')^2gh^3|D_yg|^2
+2f(f')^2g^3h|D_zh|^2
+2f^3gh|D_yg|^2|D_zh|^2.
\end{align*}
From our choice of $f,g$ and $h$, 
we see that $f''(f')^2g^3h^3 \geq 0$, 
$\langle D_y^2g\,D_yg,D_yg\rangle  \geq 0$, 
and that the last three terms are also non-negative.
Since $h$ is radial, $D_z h$ is an eigenvector of $D_z^2h$
with eigenvalue $\bar h''$. Thus
$$
\Delta_\infty \hat u \geq f^3g^3\langle D_z^2h\,D_zh,D_zh\rangle = f^3g^3\bar h'' (\bar h')^2,
$$
and
$\Delta_z h = \bar h'' + \frac{m-1}{|z|} \bar h'$, %= - \kappa h$, 
giving, since $\bar h' \leq 0$, that
$$
\bar h'' (\bar h')^2 = \left(\Delta_z h -\frac{m-1}{|z|} \bar h'\right) (\bar h')^2 \geq  (\bar h')^2 \Delta_z h
$$
whenever $|z| \neq 0$ (when $|z| = 0$ both sides of the inequality vanish). 
Hence,
$$
\Delta_\infty \hat u \geq  f^3 g^3  (\bar h')^2 \Delta_z h \geq 
f g |D\hat u|^2 \Delta_z h,
$$
where the last inequality follows since $\Delta_z h \leq 0$ and $|D_z \hat u| \leq |D\hat u|$. 
This proves claim \eqref{eq:claimforlemma} and therefore the proof of Lemma \ref{le:supersol} is complete.
\end{proof}

\bigskip

\noindent
\begin{proof}[Proof of Theorem~\ref{th:exp-decay-pipe}]
We choose $\hat u_L(x) = f(x_1) g(y) h(z)$ where
\begin{align*}
f(x_1) = A \left(\widehat\Phi(x_1) -\alpha_L\left(\frac{x_1}{L}\right)^2 + \frac{x_1}{L} \right) \quad \text{and} \quad
g(y) = 1 + \epsilon_L \sum_{i = 1}^l y_i^2,
\end{align*}
for some small $0 < \alpha_L, \epsilon_L$ to be chosen later, 
and observe that then the inequality in Lemma \ref{le:supersol} is implied in $\Omega \cap\{\mathbb{R}^n : 0 < x_1 < L\}$ by
\begin{align}\label{eq:johejsan1}   
   - \frac{\Delta_z h}{h} \geq \frac{\Phi(x,\hat u_L, D\hat u_L)}{ \lambda\, \hat u_L} +\frac{\Lambda}{\lambda}  \left( \frac{ \widehat\Phi'' - \frac{2 \alpha_L}{L^2}}{\widehat \Phi}  +  2 l \epsilon_L  \right).
\end{align}

\noindent
To find the `best' $h = h(z) = \bar h(|z|)$ we solve for positive solutions of
\begin{align}\label{eq:MM+}
\Delta_z h(z) + \kappa^+ \,h(z) = 0\quad \text{in} \quad \{z\in \mathbb{R}^m : |z| < M^+\} \quad\text{with}\quad \bar h(M^+) = 0,%\quad \text{on} \quad |z| = M^+, 
\end{align}
in which $M^+ \geq M$ will be specified later. 
The well known solutions are 
\begin{align*}%\label{eq:besselsolution}
h(z)=c |z|^{-(m-2)/2}
J_{(m-2)/2}\!\left(\sqrt{\kappa^+}|z|\right),
\end{align*}
where $c$ is chosen so that \(\bar h(0)=1\), 
and $J_\nu$ denotes the Bessel function of the first kind of order $\nu$.
The solution $h$ is $C^{\infty}$ and the boundary conditions and positivity imply that $\kappa^+$ is given by \eqref{eq:kappa} but with $M$ replaced by $M^+$.

For each \(L>0\), choose \(\alpha_L>0\) sufficiently small that
\(f''\geq0\) on \(0<x_1<L\); this is possible since
\(\widehat\Phi''>0\) on \([0,L]\). We then choose \(\epsilon_L>0\)
so that
\begin{align}\label{eq:jojo}
    \frac{ \widehat\Phi''(x_1) - \frac{2\alpha_L}{L^2}}{\widehat \Phi(x_1)}  +  2 l \epsilon_L  \leq 
      \frac{ \widehat\Phi''(x_1)}{\widehat \Phi(x_1)},
\end{align}
and since \(\widehat\Phi\) is positive, nonincreasing, and satisfies
\(\widehat\Phi(0)=1\), this inequality is equivalent to
\(2lL^2\epsilon_L\leq 2\alpha_L/\widehat\Phi(x_1)\). Thus, when
\(\ell>0\), we may take \(\epsilon_L=\alpha_L/(\ell L^2)\). 
When \(\ell=0\), we set \(\epsilon_L=0\).
Therefore, by \eqref{eq:jojo}, \eqref{eq:johejsan1} and our choice of $h$ it is enough to prove
\begin{align}\label{eq:johejsan}   
   \kappa^+ \geq \frac{\Phi(x,\hat u_L, D\hat u_L)}{ \lambda\, \hat u_L} +\frac{\Lambda}{\lambda}  \frac{ \widehat\Phi''(x_1)}{\widehat \Phi(x_1)}.
\end{align}
%

%\noindent
We now handle the lower-order terms.  
Take $A = A(\bar \epsilon, C, m, M) \geq \frac{2 C}{\bar \epsilon \bar h(M)}$ and observe that then
\begin{align}\label{eq:choice_source}
\frac{\Phi(x,\hat u_L, D\hat u_L)}{ \lambda\, \hat u_L} \leq \frac{C}{A \bar h(M)} \leq \frac{\bar \epsilon}2,
\end{align}
and thus \eqref{eq:johejsan} is implied by
\begin{align*}
\frac{\lambda(x_1)}{\Lambda(x_1)}\left(\kappa^+ - \frac{\bar \epsilon}2 \right) \geq \frac{\widehat\Phi''(x_1)}{\widehat \Phi(x_1)}. 
\end{align*}
Now we can, by taking $M^+$ close enough to $M$, ensure that $\kappa^+ - \bar\epsilon/2 \geq \kappa -\bar\epsilon$ so that it is enough to prove
\begin{align}\label{eq:ineq_theorem}
\frac{\lambda(x_1)}{\Lambda(x_1)} \left(\kappa - {\bar \epsilon} \right) \geq    \frac{\widehat\Phi''(x_1)}{\widehat \Phi(x_1)}, \quad \text{where} \quad \kappa = \left(\frac{j_{\frac{(m-2)}{2}, 1} }{M}\right)^2.
\end{align}
This is the inequality in the theorem, and we have now proved that it implies 
the inequality in Lemma \ref{le:supersol}. 
This gives us a family of strict classical supersolutions $\hat u_L$ of \eqref{eq:main} in $\Omega \cap\{\mathbb{R}^n : 0 < x_1 < L\}$, whenever $L > 0$.

We now perform the comparison argument, and to do so we add a boundary in the `open' $y$-directions if $\ell > 0$:  
For each $L > 0$ we construct the domain
\begin{align*}
&\Omega_L = \Omega \cap \{\mathbb{R}^n : 0 < x_1 < L, |y|^2 < \omega_L(x_1), |z| < M\}, \notag\\ %\quad 
&\text{where} \quad  
\omega_L(x_1) = A f(x_1)^{-1} \epsilon_L^{-1},
\end{align*}
and claim that for $A$ large enough, independent of $L$, 
\begin{align}\label{eq:bchatu}
    \hat u_L \geq C_1 \quad \text{on}\quad \partial \Omega_L \setminus \Gamma^+.%\Sigma_0 \cup \Gamma^+ \cup \Sigma_L. 
\end{align}
We assume that claim \eqref{eq:bchatu} holds for the moment 
and prove it later. 
Observe that $\omega_L \to \infty$ uniformly as $L \to \infty$, and therefore $\Omega_L \to \Omega$ in the limit $L \to \infty$.  
Moreover, since $u \leq C_1$ on $\overline \Omega$, $u\leq 0$ on $\Gamma^+$ by assumption, 
and $\hat u_L \geq 0$ on $\overline \Omega_L$ by construction, 
claim \eqref{eq:bchatu} and Lemma \ref{le:comp-weak-some-sence} allow us to compare $u$ with $\hat u_L$ for any $L > 0$. 
In particular, we conclude 
\begin{align*}
    \hat u_L \geq u \quad \text{on} \quad \partial\Omega_L \quad \text{giving}\quad  \hat u_L \geq u \quad \text{in} \quad \Omega_L.
\end{align*}
Now, since $L>0$ is arbitrary, $\Omega_L \to \Omega$ as $L \to \infty$, we have
\begin{align*}
u (x) \leq \inf_{L>x_1} \hat u_L(x) \leq  A \widehat \Phi(x_1) \quad \text{whenever}\quad x \in \Omega,
\end{align*}
which is the desired estimate.
 
To finish the proof, it only remains to establish the claim \eqref{eq:bchatu}. 
To do so, observe that 
\begin{align*}
  \left(\partial \Omega_L \setminus \Gamma^+\right) \subset \left(\Sigma_0 \cup  \Sigma_L \cup \{\mathbb{R}^n : |y|^2 = \omega_L\}\right).  
\end{align*}
Since $g \geq  1$ and $\bar h(M) \leq h$ on $\overline \Omega_L$, we have
\begin{align*}
\hat u_L(x) \geq \bar h(M) f(x_1) = A \bar h(M)\left(\widehat\Phi(x_1) - \alpha_L\left(\frac{x_1}{L}\right)^2 + \frac{x_1}{L} \right)  \quad \text{whenever} \quad  x \in \overline \Omega_L.  
\end{align*}
This gives $\hat u_L \geq A \bar h(M)/2 \geq C_1$ on $\Sigma_0 \cup \Sigma_L$ if $A$ is large enough and if we assume, as we may, 
that $\alpha_L \leq \frac{1}{2}$. 
In particular, we now take $A$ according to
\begin{align}\label{eq:choice_boundary}
A = A(\bar \epsilon, C, C_1, m, M) \geq \frac{2}{\bar h(M)} \max\left\{ C_1, \frac{C}{\bar \epsilon}\right\}.
\end{align}
It remains to ensure that $\hat u_L \geq C_1$ on $\{\mathbb{R}^n : |y|^2 = \omega_L\}$, but this follows by construction since on the set 
\begin{align*}
 |y|^2  = \omega_L(x_1) =  A f(x_1)^{-1} \epsilon_L^{-1} \quad \text{for}\quad x_1 \in  (0,L), 
\end{align*}
it holds that
\begin{align*}
\hat u_L \geq \bar h(M) f(x_1) \left(1 + \epsilon_L |y|^2\right) =  \bar h(M) f(x_1) \left(1 + \frac{A}{f(x_1)}\right) \geq C_1 \geq u. %\quad \text{on} \quad  |z| \leq M.  
\end{align*}
The proof of Theorem \ref{th:exp-decay-pipe} is complete.
\end{proof}

\bigskip

\noindent
\begin{proof}[Proof of Corollary \ref{cor:exp-decay-pipe}]
Following the proof of Theorem \ref{th:exp-decay-pipe} but with $\Phi\equiv0$ we see that we can take $\bar \epsilon = 0$ and  ignore \eqref{eq:choice_source}. 
Therefore, we can take $M^+ = M$ and thereby obtain the desired family of strict supersolutions whenever \eqref{eq:ineq_theorem} holds with $\bar\epsilon = 0$.

We next update the choice of $A$ in \eqref{eq:choice_boundary}, 
and to do so we assume in the following that $\Omega$ is a maximal $(m,M)$-thin domain since any smaller allowed domain $\Omega$ makes the situation easier. 
This time we have $\bar h(M) = 0$, i.e. the barrier vanishes on the boundary where $|z| = M$. 
The subsolution $u$ may jump immediately from 0 to $C_1$ on $\Sigma_0$, 
making it impossible to dominate by $\hat u_L$ on that set. 
Therefore, we replace it by $\Sigma_1$ and make comparison in the corresponding slightly smaller set. 
To dominate the subsolution $u$ near $|z| = M$ on $\Sigma_1$ and $\Sigma_L$ we need assumption \eqref{eq:for_corr} and the fact that the Bessel function has nonvanishing gradient near its first zero. 
In particular, there exist $\delta = \delta(m,M) > 0$ and $\hat c = \hat c(m,M) > 0$ such that
\begin{align}\label{eq:besselgradient}
    |J_\nu'(r)| \geq \hat c > 0
    \quad \text{whenever} \quad
    r \in (j_{\nu,1}-\delta,\, j_{\nu,1}+\delta).
\end{align}
By \eqref{eq:besselgradient}, assumption \eqref{eq:for_corr}, $u \leq C_1$, 
and by the construction of $\hat u_L$ it follows that there exists 
$A = A(\bar c, C_1, m, M)$ such that 
\begin{align}\label{eq:corollaryboundaryextra}
\hat u_L = f(x_1) g(y) h(z) \geq A\left( \widehat\Phi(x_1) -\alpha_L\left(\frac{x_1}{L}\right)^2 + \frac{x_1}{L} \right) h(z) \geq u \quad \text{on} \quad \Sigma_1 \cup \Sigma_L.
\end{align}
Indeed, to prove \eqref{eq:corollaryboundaryextra} on $\Sigma_1$ we bound $\widehat\Phi(1)$ from below, depending only on $m,M$: 
Since $\widehat\Phi$ is positive and convex, 
we have $\widehat\Phi'\le0$ and
$\widehat\Phi'\to0$. Multiplying
$\widehat\Phi''\le\kappa\widehat\Phi$ by $\widehat\Phi'\le0$
and integrating from $x_1$ to infinity gives
\[
    (\widehat\Phi'(x_1))^2\le\kappa\,\widehat\Phi(x_1)^2.
\]
Thus $\widehat\Phi'/\widehat\Phi\ge-\sqrt{\kappa}$,
and integration over $[0,1]$, using $\widehat\Phi(0)=1$, yields
\[
    \widehat\Phi(1)\ge e^{-\sqrt{\kappa}}.
\]
To ensure $\hat u_L \geq u$ on the part of the boundary of $\Omega_L$ where $|y|^2  = \omega_L(x_1)$ we note that on the set
\begin{align*}
 |y|^2  = \omega_L(x_1) =  A f(x_1)^{-1} \epsilon_L^{-1} \quad \text{for}\quad x_1 \in  (0,L), 
\end{align*}
it holds that
\begin{align*}
\hat u_L = f(x_1) \left(1 + \epsilon_L |y|^2\right)  h(z) = f(x_1)  \left(1 + \frac{A}{f(x_1)}\right)h(z) \geq A h(z) \geq u. %\quad \text{on} \quad  |z| \leq M.  
\end{align*}

To finish the proof it only remains to clarify that the replacement of $\Sigma_0$ with $\Sigma_1$ does not affect the result.
Indeed, now we only obtain
$
u \leq A\widehat\Phi(x_1)
$
for $x_1 \geq 1$.
%But, as we have already ensured the validity of taking $A = A(\bar c, C_1, m, M)$ so that $\hat u_L(1,0,0) =  f(1) g(0) h(0) \geq C_1$, the conclusion also holds for $0\leq x_1 < 1$ since $\widehat\Phi$, and thus also $f$, are non-increasing. 
However, by increasing $A$ if necessary so that
$A\geq C_1e^{\sqrt{\kappa}}$, we obtain $u(x)\leq C_1\leq A\widehat\Phi(1)
\leq A\widehat\Phi(x_1)$ for $0<x_1<1$. 
This completes the proof. 
\end{proof}

\bigskip

%%%%%%%%%%%%%%%%%%%%%%%%%%%%%%%%%%%%%%%%%%%%%%%%%%%%%%%%%%%%%%%%%%%%%%%%%%%%
%%%%%%%%%%%%%%%%%%%%%%%%%%%%%%%%%%%%%%%%%%%%%%%%%%%%%%%%%%%%%%%%%%%%%%%%%%%%
%%%%%%%%%%%%%%%%%%%%%%%%%%%%%%%%%%%%%%%%%%%%%%%%%%%%%%%%%%%%%%%%%%%%%%%%%%%%
%%%%%%%%%%%%%%%%%%%%%%%%%%%%%%%%%%%%%%%%%%%%%%%%%%%%%%%%%%%%%%%%%%%%%%%%%%%%
%%%%%%%%%%%%%%%%%%%%%%%%%%%%%%%%%%%%%%%%%%%%%%%%%%%%%%%%%%%%%%%%%%%%%%%%%%%%
%%%%%%%%%%%%%%%%%%%%%%%%%%%%%%%%%%%%%%%%%%%%%%%%%%%%%%%%%%%%%%%%%%%%%%%%%%%%

\section{Sharpness, examples and applications}
\label{sec:sharp}

\setcounter{theorem}{0}
\setcounter{equation}{0}

We first show that the one-dimensional condition in our main result is, in general, sharp. We then present an example involving vanishing ellipticity, discuss the effect of lower-order terms, and conclude by applying our results to $p$-subharmonic functions.

\subsection*{Sharpness}

Let $\Omega$ be a maximal $(m,M)$-thin domain, $0 < \lambda(x_1) \leq \Lambda(x_1)$ and consider the linear equation
\begin{equation} \label{eq:sharp-linear}
    -\Lambda(x_1)u_{x_1x_1}
    -\lambda(x_1)\Delta_{(y,z)} u=0
    \quad\text{in}\quad \Omega,
\end{equation}
which clearly belongs to the class of elliptic equations considered.
Assume a solution of the form
\begin{align}\label{eq:sharp}
u(x_1,z)=\widehat\Phi(x_1)h(z)
\end{align}
and observe that then the PDE separates to
\begin{align*}
    \frac{\Lambda(x_1)}{\lambda(x_1)} \frac{\widehat\Phi''(x_1)}{\widehat\Phi(x_1)} 
    = -\frac{\Delta_z h(z)}{h(z)} 
    = \kappa.
\end{align*}
Therefore, letting $h$ be as in \eqref{eq:MM+} but with $M^+ = M$ and $\kappa^+ = \kappa$, 
we conclude that $u$ is an exact solution of \eqref{eq:sharp-linear} if and only if
\begin{equation}\label{eq:sharp-ode}
    \frac{\widehat\Phi''(x_1)}
         {\widehat\Phi(x_1)}
    =
    \kappa \, \frac{\lambda(x_1)}{\Lambda(x_1)}.    
\end{equation}
Hence, whenever \eqref{eq:sharp-ode} admits a positive solution
$\widehat\Phi$ tending to zero as $x_1\to\infty$, 
the estimate in Corollary \ref{cor:exp-decay-pipe} is attained by the explicit solution \eqref{eq:sharp}, of equation \eqref{eq:sharp-linear} which belongs to the class
of PDEs considered. 
In this sense, the one-dimensional condition,
including explicit dependence on the transverse geometry via $\kappa$, 
and on ellipticity $\lambda/\Lambda$, cannot be improved.

\subsubsection*{Uniform ellipticity}

When $0<\lambda\leq\Lambda$ are constants 
\eqref{eq:sharp-ode} becomes
\[
    \widehat\Phi''
    =
    \kappa\,\frac{\lambda}{\Lambda}
    \widehat\Phi,
\]
and the corresponding positive decaying exact solution of 
\eqref{eq:sharp-linear} is 
\begin{align*}
u(x_1,z) = e^{-x_1\sqrt{\kappa\,\frac{\lambda}{\Lambda}}} h(z).
\end{align*}
This shows that the exponential decay rate
\[
    \sqrt{\kappa\,\frac{\lambda}{\Lambda}}
\]
is optimal. 
For the Laplace equation we have $\lambda=\Lambda=1$ and  
\begin{align*}
    u(x_1,z)
    = 
     e^{-\sqrt{\kappa}\,x_1} h(z)
     =
    e^{-\sqrt{\kappa}\,x_1}
    |z|^{-\frac{m-2}{2}}
    J_{\frac{m-2}{2}}
    \left(\sqrt{\kappa}\,|z|\right),
\end{align*}
so that the optimal exponential decay rate is $\sqrt{\kappa}$.

%\subsection{Examples and applications}
%\label{sec:examples}

\subsection*{Vanishing ellipticity}

Suppose that ellipticity vanishes according to
\begin{align*}
\frac{\lambda(x_1)}{\Lambda(x_1)}
    =
    \frac{\hat c}{(1+x_1)^a},
    \quad 0 < a, \quad 0 < \hat c \leq 1.   
\end{align*}
We shall see that the decay estimates will, in this case, 
show the following three different regimes, 
\begin{align}\label{eq:casesindegenerateellipticity}
u(x) \leq c
\begin{cases}
(1+x_1)^{a/4}
\exp\!\left(
-\dfrac{2\sqrt{\hat c\, \kappa}}{2-a}
(1+x_1)^{1-a/2}
\right),
& 0<a<2, \\[3mm]
(1+x_1)^{\frac{1-\sqrt{1+4\hat c\,\kappa}}{2}},
& a=2, \\[3mm]
1,
& a>2.
\end{cases}
\end{align}
Thus, stretched-exponential decay is obtained for $0<a<2$, the
borderline case $a=2$ gives power-law decay, 
while for $a>2$ we obtain no decay estimate. 
To derive these estimates from Corollary \ref{cor:exp-decay-pipe}, we proceed as follows.

%\subsubsection*{The case $0<a<2$}

\bigskip

\noindent
{\bf The case $0<a<2$}: Now
\begin{align*}
\widehat{\Phi}(x_1) = \sqrt{1+x_1}\, K_{\frac{1}{2-a}} \left( \frac{2\sqrt{\hat c \,\kappa}}{2-a} (1+x_1)^{1-a/2} \right),  
\end{align*}
where $K_\nu$ denotes the modified Bessel function of the second kind of order \(\nu\), 
solves 
\begin{align}\label{eq:ode-a-hej}
\frac{\widehat{\Phi}''}{\widehat\Phi} = \frac{\hat c\, \kappa}{(1+x_1)^a}. 
\end{align}
Using the standard asymptotic relation 
\[ 
K_\nu(z) \sim \sqrt{\frac{\pi}{2z}}\,e^{-z} \quad \text{as} \quad z\to\infty, 
\] 
we obtain the conclusion
\begin{align*}
u(x) \leq
c (1+x_1)^{a/4} \exp\left( -\frac{2\sqrt{\hat c\, \kappa}}{2-a} (1+x_1)^{1-a/2} \right) \quad\text{whenever}\quad x \in\Omega, 
\end{align*}
for some constant $c>0$. 
Hence, for $0<a<2$, the sharp decay is of stretched exponential type.

%\subsubsection*{The borderline case $a=2$}

\bigskip

\noindent
{\bf The borderline case $a=2$}: It is easy to verify that 
\begin{align*}
 \widehat\Phi(x_1) =  (1+x_1)^{-\frac{\sqrt{1+4\hat c\kappa}-1}{2}}
\end{align*}
is an exact positive solution tending to zero in this case. 
Consequently, the borderline case $a=2$ produces power-law decay.

%\subsubsection*{The case $a>2$}
\bigskip

\noindent
{\bf The case $a>2$}: Finally, assume, by way of contradiction, 
that \eqref{eq:ode-a-hej} admits a positive solution $\widehat\Phi$ satisfying
\[
    \widehat\Phi(x_1)\to0
    \quad\text{as}\quad x_1\to\infty.
\]
Since $\widehat\Phi''>0$, the function $\widehat\Phi$ is convex.
Positivity together with convergence to zero implies that
$\widehat\Phi$ is decreasing and
$\widehat\Phi'(x_1)\to0$ as $x_1\to\infty$. Integrating
\eqref{eq:ode-a-hej} from $x_1$ to infinity gives
\[
    -\widehat\Phi'(x_1)
    =
    \hat c \, \kappa
    \int_{x_1}^{\infty}
    \frac{\widehat\Phi(t)}{(1+t)^a}\,dt.
\]
Integrating once more yields
\[
    \widehat\Phi(x_1)
    =
    \hat c\, \kappa
    \int_{x_1}^{\infty}
    (t-x_1)
    \frac{\widehat\Phi(t)}{(1+t)^a}\,dt.
\]
Since $\widehat\Phi(t)\leq \widehat\Phi(x_1)$ for $t\geq x_1$ we have
\[
    1
    \leq
    \hat c\kappa
    \int_{x_1}^{\infty}
    \frac{t-x_1}{(1+t)^a}\,dt.
\]
For $a>2$,
\[
    \int_{x_1}^{\infty}
    \frac{t-x_1}{(1+t)^a}\,dt
    \leq
    \int_{x_1}^{\infty}t^{1-a}\,dt
    =
    \frac{x_1^{2-a}}{a-2}
    \rightarrow0 \quad \text{as} \quad x_1\to\infty.
\]
This is a contradiction. Thus, when $a>2$ there exists no positive solution $\widehat\Phi$ that satisfies $\widehat\Phi(x)\to 0$ as $x\to\infty$. 
We have now derived the estimates in \eqref{eq:casesindegenerateellipticity}.

As a final comment, we observe that when ellipticity vanishes at infinity, then the boundary estimates in \cite{LOT24,LOS25} may blow up ($\bar c \to \infty$ as $\lambda/\Lambda \to 0$) at infinity as well. Therefore, to ensure the assumptions in Corollary \ref{cor:exp-decay-pipe} one needs to establish \eqref{eq:for_corr} in an alternative way. 
Otherwise, without boundary estimate \eqref{eq:for_corr}, 
we have to rely on Theorem \ref{th:exp-decay-pipe} instead of Corollary \ref{cor:exp-decay-pipe}, resulting in the fact that the estimates 
are still valid, but with $\kappa$ replaced by $\kappa -\bar \epsilon$.

\subsection*{Equations with lower-order terms}

If the lower-order terms in the equation satisfy
$$
\Phi(x,s,p) \leq C\lambda(x_1)\widehat\Phi(x_1) \quad \text{whenever} \quad (x,r,p) \in \Omega\times \mathbb{R} \times \mathbb{R}^n,
$$
Theorem \ref{th:exp-decay-pipe} implies that as long as 
\begin{align}\label{eq:lowerorderode}
\sup_{x_1 > 0}\,\frac{\widehat\Phi''(x_1)}{\widehat\Phi (x_1)} \frac{\Lambda(x_1)}{\lambda(x_1)} < \kappa,  
\end{align}
then there exists a constant $A > 0$ such that 
$u(x) \leq A\, \widehat \Phi(x_1)$ whenever $x \in \Omega$. 
Thus, when $\lambda$ is constant the subsolution $u$ obeys the same decay estimate as $\Phi(x,s,p)$ in this case. 
However, if the bound $\widehat\Phi$ on the lower-order terms is better so that 
\begin{align*}
\frac{\widehat\Phi''(x_1)}{\widehat\Phi (x_1)} \geq  \kappa \, \frac{\lambda(x_1)}{\Lambda(x_1)}, 
\end{align*}
% 
%i.e. the lower-order terms decay faster than, 
%or equally fast as the optimal estimate without lower-order terms ($\widehat \Phi \equiv 0$), 
then one has to adjust $\widehat \Phi$ until it satisfies \eqref{eq:lowerorderode} (and therefore the inequality in Theorem \ref{th:exp-decay-pipe} for some $\bar\epsilon>0$) in order to obtain $u(x) \leq A\, \widehat \Phi(x_1)$. 
In such case, one may take $\bar \epsilon$ arbitrary small with the cost of $A(\bar\epsilon) \to \infty$ as $\bar\epsilon \to 0$.

As an example, suppose $\Lambda/\lambda$ is constant and that $u$ is a subsolution of 
\begin{align*}
     - \frac{\Lambda}{\lambda} u_{x_1 x_1} - \Delta_{y,z} u - Ce^{-a x_1} = 0, \quad C > 0, \quad a > 0,
\end{align*}
in an $(m,M)$-thin domain $\Omega$ satisfying $u \leq C_1$ on $\Omega$ and $u \leq 0$ on $\Gamma^+$. 
Then Theorem \ref{th:exp-decay-pipe} yields the following implications:
\begin{align*}
    a^2 <  \kappa \,\frac{\lambda}{\Lambda}
    \quad \implies \quad  
    u(x) \leq A\, e^{-a\,x_1},
\end{align*}
i.e. $u$ decays equally fast as the source term, and
\begin{align*}
    a^2 \geq \kappa\,\frac{\lambda}{\Lambda} 
     \quad \implies \quad 
     u(x) \leq A(\bar \epsilon)\, e^{-\sqrt{\frac{\lambda}{\Lambda}(\kappa-\bar\epsilon)}\,x_1}.
\end{align*}
%

%%%%%%%%%%%%%%%%%%%%%%%%%%%%%%%%%%%%%%%%%%%%%%%%%%%%%%%%%%%%%%%%%%%%%%%%%%%%
%%%%%%%%%%%%%%%%%%%%%%%%%%%%%%%%%%%%%%%%%%%%%%%%%%%%%%%%%%%%%%%%%%%%%%%%%%%%
%%%%%%%%%%%%%%%%%%%%%%%%%%%%%%%%%%%%%%%%%%%%%%%%%%%%%%%%%%%%%%%%%%%%%%%%%%%%
%%%%%%%%%%%%%%%%%%%%%%%%%%%%%%%%%%%%%%%%%%%%%%%%%%%%%%%%%%%%%%%%%%%%%%%%%%%%

\subsection*{The $p$-Laplace equation}

%\bigskip

%\noindent
%{\bf $p$-Laplace equation.}

As an application, we obtain the following estimate.

\begin{corollary}
Let $\Omega$ and $\kappa$ be as in Theorem \ref{th:exp-decay-pipe}. 
Assume $p\in(1,2]$ and that $u$ is a $p$-subharmonic function in $\Omega$, 
satisfying $u \leq C_1$ on $\overline \Omega$ and $u \leq 0$ on $\Gamma^+$. 
Then there exists $A = A(p, n, C_1, m, M) > 0$ such that 
\begin{align*}
 u(x) \leq A\,  e^{-x_1 \sqrt{\kappa\, (p-1)}}
\end{align*}
whenever $x \in  \Omega$.
\end{corollary}

\noindent
\begin{proof}
Set \(\lambda=p-1\) and \(\Lambda=1\). For \(q\neq0\),
the second term in the minimum defining
\(\mathcal P_{p-1,1}(q,X)\) is
\begin{align*}
-\operatorname{Tr}X + (2-p)\frac{\langle Xq,q\rangle}{|q|^2} = -\widetilde\Delta_p(q,X),
\end{align*}
in which the normalized $p$-Laplace operator $\widetilde\Delta_p u$ is defined via
$$
\Delta_pu
=\operatorname{div}\!\left(|Du|^{p-2}Du\right)
=|Du|^{p-2}\left(
\Delta u+(p-2)\frac{\langle D^2u\,Du,Du\rangle}{|Du|^2}
\right)
=|Du|^{p-2}\widetilde\Delta_pu.
$$
Thus, our class of PDEs covers the normalized $p$-Laplace equation. 
Moreover, from \cite[Theorem 4.1]{JJ12} it follows that any bounded viscosity supersolution of the $p$-Laplace equation, $1 < p \leq 2$, is also a weak supersolution to the same equation, and in the definition of viscosity solutions, one can omit testing at points where the gradient vanishes.
This, together with a well known comparison principle for weak sub- and supersolutions of the $p$-Laplace equation, ensures that Corollary \ref{cor:exp-decay-pipe} applies once we ensure the boundary estimate in \eqref{eq:for_corr}.
Indeed, such boundary growth rate is well known for $p$-subharmonic functions with $\bar c$ depending only on $p,n$, when the domain is convex (such as a  maximal $(m,M)$-thin domain) see \cite{AKSZ07, J15}. 
An application of Corollary \ref{cor:exp-decay-pipe} with $\Lambda/\lambda$ constant now gives
\begin{align*}
 u(x) \leq A\,  e^{-\sqrt{\frac{\lambda}{\Lambda}\kappa}\,x_1} 
 =  A e^{-x_1 \sqrt{(p-1) \kappa}}
\end{align*}
whenever $x \in  \Omega$, which is the desired estimate.
\end{proof}

%%%%%%%%%%%%%%%%%%%%%%%%%%%%%%%%%%%%%%%%%%%%%%%%%%%%%%%%%%%%%%%%%%%%%%%%%%%%
%%%%%%%%%%%%%%%%%%%%%%%%%%%%%%%%%%%%%%%%%%%%%%%%%%%%%%%%%%%%%%%%%%%%%%%%%%%%
%%%%%%%%%%%%%%%%%%%%%%%%%%%%%%%%%%%%%%%%%%%%%%%%%%%%%%%%%%%%%%%%%%%%%%%%%%%%
%%%%%%%%%%%%%%%%%%%%%%%%%%%%%%%%%%%%%%%%%%%%%%%%%%%%%%%%%%%%%%%%%%%%%%%%%%%%
%%%%%%%%%%%%%%%%%%%%%%%%%%%%%%%%%%%%%%%%%%%%%%%%%%%%%%%%%%%%%%%%%%%%%%%%%%%%
%%%%%%%%%%%%%%%%%%%%%%%%%%%%%%%%%%%%%%%%%%%%%%%%%%%%%%%%%%%%%%%%%%%%%%%%%%%%


\begin{thebibliography}{A}

\bibitem{AKSZ07}
{\sc Aikawa H., Kilpeläinen T., Shanmugalingam N., Zhong X.,}
\emph{Boundary Harnack principle for p-harmonic functions in smooth Euclidean domains},
Potential Analysis 26.3 (2007), 281--301.

\bibitem{All99}
{\sc Allaire G., Amar M.,}
\emph{Boundary layer tails in periodic homogenization},
ESAIM: Control, Optimisation and Calculus of Variations 4 (1999), 209--243.

%ej citerad
%\bibitem{AV24}
%{\sc Arrieta J. M., Villanueva-Pesqueira M.,}
%\emph{Elliptic and parabolic problems in thin domains with doubly weak oscillatory boundary},
%arXiv preprint arXiv:2405.05607 (2024).

\bibitem{BFMW25}
{\sc Bomba E., Fabricius J., Manjate S., Wall P.,}
\emph{Pressure-driven flow in thin straight tubes of non-uniform cross-section},
ZAMM -- Journal of Applied Mathematics and Mechanics / Zeitschrift für Angewandte Mathematik und Mechanik 105 (2025).

\bibitem{SR86}
{\sc Breuer S., Roseman J. J.,}
\emph{Phragmén-Lindelöf decay theorems for classes of nonlinear Dirichlet problems in a circular cylinder},
Journal of Mathematical Analysis and Applications 113.1 (1986), 59--77.

%ej citerad, kan vara relevant för Growth2 artikeln
%\bibitem{CGR10}
%{\sc Cappiello M., Gramchev T., Rodino L.,}
%\emph{Entire extensions and exponential decay for semilinear elliptic equations},
%Journal d'Analyse Mathématique 111.1 (2010), 339--367.

% ej citerad
%\bibitem{CDV07} {\sc Capuzzo Dolcetta I., Vitolo A.,}
%\emph{A qualitative Phragm\'en-Lindel\"of theorem for fully nonlinear elliptic equations}, J. Differential Equations 243 %(2007), 578--592.

\bibitem{CDV19} {\sc Capuzzo Dolcetta I., Vitolo A.,}
\emph{Directional ellipticity on special domains: Weak Maximum and Phragm\'en-Lindel\"of principles}, Nonlinear Analysis 184 (2019), 69--82.

\bibitem{C07}
{\sc Casado-Díaz J.,}
\emph{Exponential decay for the solutions of nonlinear elliptic systems posed in unbounded cylinders},
Journal of Mathematical Analysis and Applications 328.1 (2007), 151--169.

%ej citerad, kan vara relevant för Growth2 artikeln
%\bibitem{CDM04}
%{\sc Cavalcanti M. M., Domingos Cavalcanti V. N., Ma T. F.,}
%\emph{Exponential decay of the viscoelastic Euler-Bernoulli equation with a nonlocal dissipation in general domains},
%(2004), 495--510.

\bibitem{CR02}
{\sc Chipot M., Rougirel A.,}
\emph{On the asymptotic behaviour of the solution of elliptic problems in cylindrical domains becoming unbounded},
Communications in Contemporary Mathematics 4.01 (2002), 15--44.

\bibitem{CIL92}
{\sc Crandall M. C., Ishii H., Lions P.-L.,}
\emph{User's guide to viscosity solutions of second order partial differential equations},
Bulletin of the American Mathematical Society 27.1 (1992), 1--67.

\bibitem{DNG12}
{\sc Delitsyn A. L., Nguyen B. T., Grebenkov D. S.,}
\emph{Exponential decay of Laplacian eigenfunctions in domains with branches of variable cross-sectional profiles},
The European Physical Journal B 85.11 (2012), 371.

\bibitem{G11}
\sc{Galdi G. P.,}
\emph{An Introduction to the Mathematical Model Theory of the Navier–Stokes Equations}, 2nd edition,
Springer, New York (2011).

\bibitem{Hol13}
{\sc Holmes M. H.,}
\emph{Introduction to perturbation methods},
2nd ed., Texts in Applied Mathematics 20, Springer, New York (2013).

\bibitem{J15}
{\sc Julin V.,}
\emph{Generalized Harnack inequality for nonhomogeneous elliptic equations},
Archive for Rational Mechanics and Analysis 216 (2015), 673--702.

\bibitem{JJ12}
{\sc Julin V., Juutinen P.,}
\emph{A new proof for the equivalence of weak and viscosity solutions for the p-Laplace equation},
Communications in Partial Differential Equations 37.5 (2012), 934--946.

\bibitem{Know66}
{\sc Knowles J. K.,}
\emph{On Saint-Venant's principle in the two-dimensional linear theory of elasticity},
Archive for Rational Mechanics and Analysis 21 (1966), 1--22.

\bibitem{L57}
{\sc Lax P. D.,}
\emph{A Phragmén-Lindelöf theorem in harmonic analysis and its application to some questions in the theory of elliptic equations},
Communications on Pure and Applied Mathematics 10.3 (1957), 361--389.

%ej citerad, kan vara relevant för Growth2 artikeln
%\bibitem{L24}
%{\sc Le Balch K.,}
%\emph{On the exponential decay for real-valued solutions to elliptic equations with singular potentials in the plane},
%Journées Équations aux Dérivées Partielles (2024), 1--14.

\bibitem{L22}
{\sc Lundström N. L. P.,}
\emph{Growth of subsolutions to fully nonlinear equations in halfspaces},
Journal of Differential Equations 320 (2022), 143--173.

\bibitem{LOS25}
{\sc Lundström N. L. P., Olofsson M., Singh J.,}
\emph{Estimates for viscosity solutions of fully nonlinear equations near smooth boundaries},
arXiv preprint arXiv:2506.12477 (2025).

\bibitem{LOT24}
{\sc Lundström N. L. P., Olofsson M., Toivanen O.,}
\emph{Strong maximum principle and boundary estimates for nonhomogeneous elliptic equations},
Potential Analysis 60 (2024), 425--443.

\bibitem{Ole80}
{\sc Oleinik O. A., Iosif'yan G. A.,}
\emph{Behavior at infinity of the solutions of the Neumann problem for a second-order elliptic equation in an unbounded domain},
Uspekhi Matematicheskikh Nauk 35.4 (1980), 197--198.

\bibitem{bok_spec_func}
{\sc Olver F. W. J.,}
\emph{NIST handbook of mathematical functions},
Cambridge University Press, Cambridge (2010).

\bibitem{Pan98}
{\sc Panasenko G. P.,}
\emph{Asymptotic expansion of the solution of Navier-Stokes equation in a tube structure},
C. R. Acad. Sci. Paris, Série IIb 326 (1998), 867--872.

\bibitem{Pan24}
{\sc Panasenko G. P., Pileckas K.,}
\emph{Multiscale analysis of viscous flows in thin tube structures},
Advances in Mathematical Fluid Mechanics, Birkhäuser/Springer, Cham (2024).

\bibitem{Toup65}
{\sc Toupin R. A.,}
\emph{Saint-Venant's principle},
Archive for Rational Mechanics and Analysis 18 (1965), 83--96.

\end{thebibliography}
\end{document}